\documentclass[a4paper,10pt]{amsart}
\usepackage{amssymb,bm}
\usepackage{hyperref}
\usepackage[arrow,matrix]{xy}
\allowdisplaybreaks

\newcommand {\Hom}{\operatorname {Hom}}

\newcommand {\im}{\operatorname {Im}}

\newcommand {\id}{\operatorname {id}}
\newcommand {\der}{\operatorname {Der}}
\newcommand {\sgn}{\operatorname {sgn}}

\newcommand {\PH}{\operatorname {PH}}

\newcommand {\de}{\operatorname {d}}

\begin{document}
\setlength{\baselineskip}{1.4em}
\numberwithin{equation}{section}
\newtheorem{thm}{Theorem}[section]
\newtheorem{prop}[thm]{Proposition}
\newtheorem{lem}[thm]{Lemma}
\newtheorem{cor}[thm]{Corollary}
\theoremstyle{definition}
\newtheorem{defn}[thm]{Definition}
\newtheorem{rmk}[thm]{Remark}
\newtheorem{exam}[thm]{Example}

\title[Twisted Poincar\'{e} duality between Poisson
homology and cohomology]{Twisted Poincar\'{e} duality
between Poisson homology and Poisson cohomology}
\author{J.~Luo}
\address{School of Mathematical Sciences, Fudan University, Shanghai 200433, China}
\email{13110180056@fudan.edu.cn}
\author{S.-Q.~Wang}
\address{Department of Mathematics, East China University of Science and Technology, Shanghai 200237, China}
\email{sqwang@ecust.edu.cn}
\author{Q.-S.~Wu}
\address{School of Mathematical Sciences, Fudan University, Shanghai 200433, China}
\email{qswu@fudan.edu.cn}

\begin{abstract}
A version of the twisted Poincar\'{e} duality is proved between the
Poisson homology and cohomology of a polynomial Poisson algebra with
values in an arbitrary Poisson module. The duality is achieved by
twisting the Poisson module structure in a canonical way, which is
constructed from the modular derivation. In the case that the Poisson
structure is unimodular, the twisted Poincar\'{e} duality reduces to
the Poincar\'{e} duality in the usual sense. The main result generalizes the
work of Launois-Richard \cite{LR} for the quadratic Poisson
structures and Zhu \cite{Zhu} for the linear Poisson structures.
\end{abstract}



\maketitle

\setcounter{section}{-1}
\section{Introduction}

Poisson structures naturally appear in classical/quantum mechanics,
in the deformation theory of commutative algebras and in
mathematical physics. They play an important role in Poisson
geometry, in algebraic geometry and non-commutative geometry.
Poisson cohomologies are important invariants of Poisson structures.
They are closely related to Lie algebra cohomology, but Poisson
cohomology is finer in general. The set of Casimir elements of the
Poisson structure is the $0$-th cohomology; Poisson derivations
modulo Hamiltonian derivations is the $1$-st cohomology. Poisson
cohomology appears as one considers deformations of Poisson
algebras. The basic concepts of Hamiltonian mechanics are
conveniently expressed in terms of Poisson cohomology.

The Poisson cohomology of a symplectic manifold is precisely its de
Rham cohomology. The Poisson cohomology of a linear Poisson
structure is also easy to understand. The Poisson cohomology for all
Poisson structures on affine plane which are homogeneous is computed
in \cite{RV} (see also \cite{Mon}). Pichereau \cite{Pic} succeeded
to compute the Poisson cohomology for affine Poisson 3-dimensional
spaces and Poisson surfaces defined by a weight homogeneous polynomial
with an isolated singularity. In general, computing the Poisson
cohomology of a given Poisson structure is very difficult.

Poisson homology has a close relation with the Hochschild homology
of its deformation algebra. In some cases, Poisson homology is
computable. Van den Bergh \cite{VdB}, Marconnet \cite{Mar},  Berger
and Pichereau \cite{BP} considered some 3-dimensional Calabi-Yau
algebras as deformations of polynomial Poisson algebras
respectively, by computing the corresponding Poisson homology, they
computed the Hochschild homology of the deformation algebras by
using the Brylinsky spectral sequence \cite{Bry,EG}. In a similar
way, Tagne Pelap computed the Hochschild homology of 4-dimensional
Sklyanin algebras \cite{TaP1,TaP2}. The Poisson structures
considered by them are all Jacobian Poisson structure (JPS, for
short) \cite{Prz}. JPS is a unimodular Poisson structure, and there
is a Poincar\'{e} duality between Poisson homology and cohomology
for unimodular Poisson algebras.  Xu \cite{Xu}
proved this for unimodular Poisson manifolds. So, by the duality,
the Poisson cohomology in these cases can also be obtained. Tagne Pelap \cite{TaP3}
also computed the Poisson (co)homology for some non-unimodular Poisson structures (GJPS) over the
polynomial algebras in $3$ variables.

On the other hand, one can compute the Poisson (co)homology via the
Hochschild (co)homology of the deformation algebra \cite{LR}. By
using the semiclassical limit of the dualizing bimodule between the
Hochschild homology and Hochschild cohomology of the corresponding
quantum affine space, Launois-Richard \cite{LR} obtained a new
Poisson module which provides a twisted Poincar\'{e} duality between
Poisson homology and cohomology  for the polynomial algebra with
quadratic Poisson structure. Following this idea, Zhu \cite{Zhu}
constructed a new Poisson module structure, which comes from the
dualizing bimodule resulting in the Poincar\'{e} duality between
Hochschild homology and cohomology over the universal enveloping
algebra \cite{Ye},  and obtained a twisted Poincar\'{e} duality for
linear Poisson algebras.

The major work in this paper is to prove a version of the twisted
Poincar\'{e} duality between the Poisson homology and cohomology for
any polynomial Poisson algebras with values in an arbitrary Poisson
module. The following is Theorem \ref{main-theorem}.

\noindent
{\bf Theorem.}  Let $R$ be a polynomial Poisson algebra in $n$ variables, $M$ be a right Poisson $R$-module. Then
for all $p \in \mathbb{N}$,  there is a functorial isomorphism
$$\PH_p(R, M_t)\cong \PH^{n-p}(R, M),$$ where
$\PH^{n-p}(R, M)$ is the $(n-p)$-th Poisson cohomology of $R$ with values in $M$,
and $\PH_p(R, M)$ is the $p$-th Poisson homology of $R$ with values in the twisted Poisson module $M_t$
twisted by the modular derivation of $R$ (see Proposition \ref{twist-poisson-module-structure}).

The isomorphism is obtained from an explicit isomorphism between the Poisson cochain complex of $R$ with values in $M$ and
the Poisson chain complex of $R$ with values in $M_t$.
This generalizes the results in \cite{LR} and \cite{Zhu}. In the
case that the Poisson structure is unimodular, the twisted
Poincar\'{e} duality reduces to the Poincar\'{e} duality in usual
sense. In a subsequent paper, a version of the twisted Poincar\'{e} duality between the Poisson homology and cohomology for
any smooth affine Poisson variety with trivial canonical bundle is proved \cite{LWW}.
A general duality theorem is proved by Huebschmann \cite{Hue1}
in the setting of Lie-Rinehart algebra.

The basic definitions concerning Poisson (co)homology are recalled
in Section 1. The modular derivation and log-Hamiltonian derivation
are discussed in Section 2. One example is given for a Polynomial
Poisson algebra with Hamiltonian modular derivation such that the
Poincar\'{e} duality between the Poisson homology and Poisson
cohomology does not hold. So, the definition of unimodular Poisson
structure is modified to be one with log-Hamiltonian modular
derivation. The twisted Poisson module structure by a Poisson
derivation is also given in Section 2. The main result is proved in
Section 3.

\section{Preliminaries and definitions}\label{sec:def.prelimi}
Let $k$ be a field of characteristic zero. All vector spaces and algebras are over $k$.

\subsection{Poisson algebras and modules}
\begin{defn} \cite{Lic, Wei0} A commutative $k$-algebra $R$ equipped with a
bilinear map $\{-, -\} : R \times R \to R$ is called a {\bf Poisson
algebra} if
\begin{enumerate}
\item $R$ with$\{-, -\} : R \times R \to R$  is a $k$-Lie algebra;
\item $\{-, -\} : R \times R \to R$  is a derivation in each argument with respect to the multiplication of $R$.
\end{enumerate}
\end{defn}

\begin{exam}
\noindent (1) [Poisson's classical bracket, 1809] Let $R = C^\infty(\mathbb{R}^{2})$ be
the ring of smooth functions over the real plane $\mathbb{R}^2$. For
any $f, g \in R$, let $\{f , g\} =\frac{\partial f} {\partial x}
\frac{\partial g} {\partial y} - \frac{\partial g}{\partial x}
\frac{\partial f} {\partial y}$. Then $R$ becomes a Poisson algebra.

\noindent
(2) \cite{GMP} Let $R=k[x_1, x_2, \cdots, x_n]$ be a polynomial algebra. Given
any $n-2$ polynomials $f_3, \cdots, f_n \in R$, and $u \in R$,
defining $\{-, -\}$ on $R$ by
$$\{f_1, f_2\}=u \frac{\de\!f_1 \wedge \de\!f_2 \wedge \de\!f_3 \wedge \cdots \wedge \de\!f_n}{\de\!x_1 \wedge \de\!x_2 \wedge \cdots \wedge
\de\!x_n} = uJ(f_1,f_2,f_3,\cdots,f_n),$$ where
$J(f_1,f_2,\cdots,f_n)=\det(\frac{\partial f_i}{\partial x_j})_{n
\times n}$ is the Jacobian determinant of $f_1, f_2, \cdots, f_n$.
Then $R$ is a Poisson algebra and it is called a {\bf generalized
Jacobian Poisson Structure} (GJPS) over $R$ given by $f_3, \cdots,
f_n$ and $u$. If $u=1$, then the Poisson structure is called a {\bf
Jacobian Poisson Structure} (JPS) given by $f_3, \cdots, f_n$.


%

\noindent
(3) Let $A$ be an algebra which is not necessarily commutative and $z\in
A$ be a non zero-divisor central element in $A$ such that $\bar{A}=A/zA$ is
commutative. For any $\bar{x}+zA, \bar{y}+zA \in \bar{A} \, (x, y \in A)$,
define
$$\{\bar{x}, \bar{y}\} = [x, y]/z +z A \in \bar{A}.$$
Then $\bar{A}$ with $\{,\}$ is a Poisson algebra. Sometimes,
$\bar{A}$ is called a semiclassical limit of $A$, $A$ is called a
quantization of $\bar{A}$ and  $A_q=A/(z-q)A$ is called a
deformation of $\bar{A}$, where $q \in k^*$.
\end{exam}

\begin{defn} \label{poisson-module}
 \cite{Oh} A right {\bf Poisson module} $M$ over the Poisson algebra $R$ is a k-vector
space $M$ endowed with two bilinear maps $\cdot$ and $\{-, -\}_M : M
\times R \to M $ such that
 \begin{enumerate}
\item $(M, \cdot)$ is a  module over the commutative algebra $R$;
\item $(M, \{-, -\}_M)$ is a right Lie-module over the Lie algebra $(R, \{-, -\})$;
\item $\{xa, b\}_M = \{x, b\}_M a + x\{a, b\}$ for any $a, b \in R$ and $x \in M$;
\item $\{x, ab\}_M = \{x, a\}_M  b + \{x, b\}_M a$ for any $a, b \in R$ and $x \in M$.
\end{enumerate}
\end{defn}

Left Poisson modules are defined similarly. Any Poisson
algebra $R$ is naturally a right or left Poisson module over itself.

\subsection{Poisson homology and cohomology}

Let M be a right Poisson module over the  Poisson algebra $R$. In
general, let $\Omega^{1}(R)$ be the K\"{a}hler differential module
of $R$ and $\Omega^{p}(R)$ be the $p$-th K\"{a}hler differential
forms. There is a  canonical chain complex
\begin{equation}
 \label{poison-chain-complex}
\cdots \longrightarrow M \otimes_R
\Omega^{p}(R)\stackrel{\partial}{\longrightarrow} M \otimes_R
\Omega^{p-1}(R)\stackrel{\partial}{\longrightarrow}\cdots
\stackrel{\partial}{\longrightarrow}M \otimes_R \Omega^{1}(R)
\stackrel{\partial}{\longrightarrow} M \otimes_R R \to 0
\end{equation}
where $\partial \colon M \otimes_R \Omega^{p}(R)\longrightarrow M
\otimes_R \Omega^{p-1}(R)$ is defined as:
\begin{multline*}
 \partial(m \otimes \de\!a_1\wedge \cdots \wedge \de\!a_p)=\sum_{i=1}^p(-1)^{i-1}\{m, a_i\}_M\otimes
  \de\!a_1\wedge \cdots \widehat{\de\!a_i} \cdots \wedge \de\!a_p \\
  {}+\sum_{1\le i<j\le p}(-1)^{i+j}m \otimes \de \{a_i, a_j\}
  \wedge \de\!a_1\wedge \cdots  \widehat{\de\!a_i} \cdots  \widehat{\de\!a_j} \cdots \wedge
  \de\!a_p.
\end{multline*}

\begin{defn} \cite{Mas}
The complex \eqref{poison-chain-complex} is called the {\bf Poisson
complex} of $R$ with values in $M$, and its $p$-th homology is
 called the $p$-th {\bf Poisson homology} of $R$ with values in
$M$, denoted by $\PH_p(R, M)$.
\end{defn}

In the case $M=R$, the Poisson homology is the original definition
of canonical homology given by Brylinsky \cite{Bry}.

Obviously, $\PH_0(R, M)=M/<\{\{m,a\}_M \mid m \in M, a \in R\}>$ and
$\PH_1(R, M)=\{ \sum_i m_i \otimes \de\!a_i \in M \otimes
\Omega^1(R)\mid \sum_i\{m_i, a_i\}=0 \}/<\{\{m,a\} \otimes
\de\!b-\{m, b\} \otimes \de\!a - m \otimes \de\!\{a, b\} \mid m \in
M, a, b \in R\}>$. Indeed, there is no good interpretation for
$\PH_0(R, M)$ and $\PH_1(R, M)$ as the Poisson cohomology which will
be defined shortly.

Next, let $\mathfrak{X}^p(M)$ be the set of all skew-symmetric $p$-fold
$k$-linear maps from $R$ to $M$ which are derivation in each
argument, that is,
$$\mathfrak{X}^p(M)=\{F \in \Hom_k(\wedge^p R, M) \mid F \, \textrm{is a derivation in each argument}\}.$$

Obviously, $\mathfrak{X}^0(M)=M$ and $\mathfrak{X}^1(M)=\der_k(R,
M)$. There is a canonical cochain complex£º
\begin{equation}
 \label{poison-cochain-complex}
0 \longrightarrow M \stackrel{\delta}{\longrightarrow} \mathfrak{X}^1(M)
\stackrel{\delta}{\longrightarrow} \cdots
\stackrel{\delta}{\longrightarrow} \mathfrak{X}^p(M)
\stackrel{\delta}{\longrightarrow} \mathfrak{X}^{p+1}(M) \longrightarrow
\cdots
\end{equation}
where $\delta \colon \mathfrak{X}^p(M)\longrightarrow \mathfrak{X}^{p+1}(M)$ is
defined as $F \mapsto \delta(F)$ with
\begin{multline*}
  \delta(F)(a_1\wedge \cdots \wedge a_p \wedge a_{p+1})
  =\sum_{i=1}^{p+1}(-1)^{i}\{F(a_1 \wedge \cdots \wedge \widehat{a_i}\wedge \cdots \wedge a_{p+1}), a_i\}_M \\
  {}+\sum_{1\le i<j\le p+1}(-1)^{i+j}
  F(\{a_i, a_j\}\wedge a_1\wedge \cdots \wedge \widehat{a_i}\wedge \cdots \wedge \widehat{a_j}\wedge \cdots \wedge
  a_{p+1}).
\end{multline*}

\begin{defn} \cite{Lic, Hue}
The complex \eqref{poison-cochain-complex} is called the {\bf
Poisson cochain complex} of $R$ with values in $M$, and its $p$-th
cohomology is called the $p$-th {\bf Poisson cohomology} of $R$ with
values in $M$, denoted by $\PH^p(R, M)$.
\end{defn}

Note $\delta^0: M \to \mathfrak{X}^1(M)$ is the map $m \mapsto -\{m,
-\}_M$ and $\delta^1: \mathfrak{X}^1(M) \to \mathfrak{X}^2(M)$ is
the map $F \mapsto \delta^1(F)$ with $\delta^1(F)(a \wedge
b)=-\{F(b), a\}_M + \{F(a), b\}_M - F(\{a, b\})$. The elements in
$\ker \delta^1$ are called {\bf Poisson derivations}, and the
elements $\{m, -\}_M \in \im \delta^0$  are called {\bf Hamiltonian
derivations}. Hence $\PH^0(R, M)=\{m \in M \mid \{m, a\}_M=0,
\forall\, a \in R\}$ consisting of all Casimir elements in $M$, and
$\PH^1(R, M)=\mbox{\{Poisson derivations\}}/\mbox{\{Hamiltonian
derivations\}}$.

\section{ Modular class and unimodular Poisson structure }

\subsection{Contraction map}
Let $R$ be a commutative algebra.
\begin{defn}\label{ctr2}
For any $P\in \mathfrak{X}^{p}(R)$, the {\bf contraction map}
$\iota_{P}: \Omega^{*}(R)\rightarrow \Omega^{*}(R)$ is a graded
$R$-linear map of degree $-p$, which is defined as $\iota_{P}:
\Omega^{k}(R)\rightarrow \Omega^{k-p}(R)$: when $k < p, \iota_{P}=
0$; when $k\geq p$ and $\omega =\de\! a_1\wedge \de\!
a_2\wedge\cdots\wedge\de\! a_k \in \Omega^{k}(R)$,
$$
\iota_{P}(\omega)= \sum\limits_{\sigma\in S_{p,k-p}}
 \sgn(\sigma) P(a_{\sigma(1)},\cdots, a_{\sigma(p)})
 \de\! a_{\sigma(p+1)}\wedge \de\! a_{\sigma(p+2)}\wedge\cdots\wedge\de\! a_{\sigma(k)},$$
where $S_{p, k-p}$ denotes the set of all $(p,k-p)$-shuffles, which
are the permutations $\sigma \in S_k$ such that $\sigma(1)< \cdots <
\sigma(p)$ and $\sigma(p+1)< \cdots < \sigma(k)$.
\end{defn}
\noindent For $P=a\in R = \mathfrak{X}^{0}(R)$, the contraction map has to
be understood as $\iota_{P}(\omega) = a \omega$.

Similar to the universal property of the K\"{a}hler differential
module $\Omega^{1}(R)$ of $R$, it is easy to see that
$\mathfrak{X}^{k}(R) \cong \Hom_R(\Omega^{k}(R), R)$ as left
$R$-modules for any $k$.
\begin{defn}\label{ctr3}
For any $Q \in \Omega^{p}(R)$, the {\bf contraction map} $\iota_Q:
\mathfrak{X}^{*}(R)\rightarrow \mathfrak{X}^{*}(R)$ is a graded
$R$-linear map of degree $-p$, which is defined as $\iota_Q:
\mathfrak{X}^{k}(R)\rightarrow \mathfrak{X}^{k-p}(R)$: when $k < p,
\iota_{Q}= 0$; when $k\geq p$ and $F \in \mathfrak{X}^{k}(R)$,
$$(\iota_Q F) (\de\! a_1 \wedge \de\! a_2 \wedge \cdots \wedge\de\! a_{k-p})
=F(\de\! a_1\wedge \de\! a_2\wedge\cdots\wedge\de\! a_{k-p} \wedge Q
).$$
\end{defn}

From Definitions \ref{ctr2} and  \ref{ctr3}, for any $P\in
\mathfrak{X}^{p}(R)$, $f\in R$ and $\omega\in \Omega^k(R)$,
\begin{equation} \label{ctr-equality}
\iota_P(\omega \wedge \de\! f)= \iota_{P}(\omega)\wedge \de\! f +
 (-1)^{k-p+1}\iota_{\iota_{\de\!f}(P)}(\omega).
\end{equation}

\subsection{Modular class}
Let $R$ be a smooth algebra of dimension $n$ with trivial canonical
bundle $\Omega^n(R)=R \eta$ where $\eta$ is a volume form of $R$.
Then $ \eta \wedge \de\! f=0$ and 
\begin{equation} \label{formula-1}
 P(f)\eta=\de\! f\wedge \iota_{P}(\eta) \, (\forall \, P \in
 \mathfrak{X}^1(A)).
\end{equation}
If further, $R$ is a Poisson algebra, with the Poisson structure
$\pi =\{-,-\}: R \wedge R \to R$, then
\begin{equation}\label{formula-2}
\iota_{H_f}(\eta)=-\de\! f\wedge \iota_{\pi}(\eta)
\end{equation}
where $H_f:=\{f,-\}: R \to R$ is the Hamiltonian derivation associated to $f$.

%
%

Now we can define the modular derivation for smooth Poisson algebras with trivial canonical bundle.

\begin{defn}\label{modular4} Let $R$ be a smooth Poisson algebra of
dimension $n$ with trivial canonical bundle $\Omega^n(R)=R \eta$
where $\eta$ is a volume form. The {\bf modular derivation} of $R$
with respect to $\eta$ is defined as the map $\phi_{\eta}: R \to R$
such that for any $f\in R$,
$$\phi_{\eta}(f)=\frac{L_{H_f}\eta}{\eta},$$ where
$L_{H_f}=[\de, \iota_{H_f}]$ is the Lie-derivation on $\Omega^\cdot(R)$.
\end{defn}

It follows from \eqref{formula-1} and  \eqref{formula-2} that
$\phi_\eta$ is not only a derivation, but also a Poisson derivation.

Before we give the definition of the modular class for Poisson algebras, let us say
a few words geometrically. Let $P$ be a Poisson manifold with
Poisson tensor $\pi$, and with a volume form $\mu$ (for example,
when $P$ is a real orientable manifold). The operator $\phi_\mu : f
\mapsto (L_{H_f}\mu)/\mu$ is a Poisson derivation and hence a vector
field, where $H_f=\{f, -\}$ is the Hamiltonian vector field of $f$
and $L_\xi =[\de, \iota_\xi]$ is the Lie derivative with respect to
the vector field $\xi$. $\phi_\mu$ is called the {\bf modular vector
field} of $P$. When the volume form changes, the corresponding
modular vector field differs by a Hamiltonian vector field. So its
class is a well-defined element in the $1$-st Poisson
cohomology, which is called the {\bf modular class} of $P$. If the
modular class vanishes, i.e., the modular vector field is a
Hamiltonian vector field, then the Poisson manifold is called {\bf
unimodular} \cite{Wei}. Xu \cite{Xu} proved a duality between the Poisson homology and
cohomology in the case of unimodular Poisson manifold.

In algebra case, when the volume form is changed, the corresponding Poisson
derivation is modified by a so called log-Hamiltonian derivation as we see in the following.
If $\lambda$ is another volume form of $R$, then $\lambda = u \eta$ for some
unit $u \in R$, and
\begin{align*}
\phi_{\lambda}(f)\lambda = & L_{H_f}(u \eta) =\de (\iota_{H_f}(u \eta))
= \de (u \, \iota_{H_f}(\eta))\\
=&\de\! u \wedge \iota_{H_f}(\eta) + u \, \de\! \iota_{H_f}(\eta)\\
=&H_f(u) \eta  + u \, \phi_{\eta}(f)\eta\\
=&u^{-1} H_f(u)\lambda + \phi_{\eta}(f)\lambda.
\end{align*}
It follows that $\phi_{\lambda} = \phi_{\eta} - u^{-1}H_u.$ The
Poisson derivation $ u^{-1}\{u, -\}: R \to R$ is called a {\bf
log-Hamiltonian derivation} of $R$ determined by the invertible
element $u$. The {\bf modular class} of $R$ is defined as the class
$\phi_{\eta}$ modulo log-Hamiltonian derivations (see \cite{Do}). If
the modular class is trivial, i.e., $\phi_{\eta}$ is a
log-Hamiltonian derivation, then $R$ is called {\bf unimodular}. If
the only invertible elements of $R$ are non-zero elements in $k 1_R$
as in polynomial algebras, then $R$ is unimodular in this sense if
and only if the modular derivation is zero.

\subsection{Polynomial Poisson algebras}

Now, we  consider polynomial Poisson algebras. Let $R=k[x_1, x_2,
\cdots, x_n]$ be a polynomial Poisson algebra with Poisson bracket
$\{-, -\}.$ Then  $\Omega^1(R)=\oplus_{i=1}^n R \de\! x_i$ and
$\Omega^n(R)=R \de\! x_1 \wedge \de\! x_2 \wedge \cdots \wedge \de\!
x_n$ with a volume form $\eta=\de\! x_1 \wedge \de\! x_2 \wedge
\cdots \wedge \de\! x_n$. We can describe the modular derivation
$\phi_{\eta}$ easily.
\begin{lem} \label{m-class-poly} Let $R=k[x_1, x_2,
\cdots, x_n]$ be a polynomial Poisson algebra with Poisson bracket
$\{-, -\}.$ The modular derivation $\phi_\eta$ is given by
 $$\phi_\eta(f)=\sum_{j=1}^n \frac{\partial \{f,
x_j\}}{\partial x_j}, \, \forall \,  f \in R.$$
\end{lem}
\begin{proof} By definition,
\begin{align*}
\phi_\eta(f)\eta&=L_{H_f}\eta=\de (\iota_{H_f}\de\! x_1 \wedge \de\!
x_2 \wedge \cdots \wedge \de\! x_n)\\
=&\sum_{j=1}^n (-1)^{j-1} \de (H_f(x_j) \de\! x_1 \wedge \cdots
\widehat{\de\! x_j} \cdots \wedge \de\! x_n)\\
=&\sum_{j=1}^n (-1)^{j-1} \de\! (\{f, x_j\} \de\! x_1 \wedge \cdots
\widehat{\de\! x_j} \cdots \wedge \de\! x_n)\\
=&\sum_{j=1}^n (-1)^{j-1} \sum_{l=1}^n \frac{\partial \{f,
x_j\}}{\partial x_l} \de\! x_l \wedge \de\! x_1 \wedge \cdots
\widehat{\de\! x_j} \cdots \wedge \de\! x_n\\
=&\sum_{j=1}^n \frac{\partial \{f, x_j\}}{\partial x_j}  \de\! x_1 \wedge \de\! x_2 \wedge \cdots \wedge \de\! x_n\\
=&\sum_{j=1}^n \frac{\partial \{f, x_j\}}{\partial x_j} \eta.
\end{align*}
It follows that $\phi_\eta(f)=\sum_{j=1}^n \frac{\partial \{f,
x_j\}}{\partial x_j}$.
\end{proof}

There is no non-trivial Poisson structure on $k[x]$.  The
Poisson structure on $k[x, y]$ is given by $\{x, y\} = h \in
k[x, y]$; and for any fixed polynomial $ h \in k[x, y]$, $\{x,
y\} = h$ gives a Poisson structure on $k[x, y]$.  If the
modular derivation of $R$ is a Hamiltonian derivation, then there is
some $g \in k[x, y]$, such that
\begin{align*}
\phi(x)&=\partial \{x,y\}/ {\partial y}= \partial h/ {\partial
y}=\{g, x\} =  \partial g/ {\partial y}\{y, x\},\\
\phi(y)&=\partial \{y,x\}/ {\partial x}=-
\partial h/ {\partial x}=\{g, y\} = \partial g/ {\partial x}\{x,
y\},
\end{align*}
that is, $$\partial h/ {\partial y}=-h \partial g/ {\partial
y}, \, \,
\partial h/ {\partial x}=- h \partial g/ {\partial x}.$$
Then $h$ should be a constant in $k$. If $h$ is a constant, the
modular derivation of $R$ is $0$. For polynomial algebras in three
variables, this is not the case.

\begin{exam}\label{non-unimodular-example}
Let $R=k[x, y, z]$ with the  Poisson structure given by
$$\{x, y\}=0, \{y, z\}=y, \{z, x\}=-1.$$
In fact, $\{\{x, y\}, z\} + \{\{y, z\}, x\} + \{\{z, x\}, y\} =\{y,
x\}=0$ implies the Jacobi identity. The modular derivation of $R$
with respect to the volume form $\eta = \de\! x \wedge \de\! y \wedge \de\! z$
is
$$\phi_\eta =\{-, x\}=\{-x, -\},$$
which is a Hamiltonian derivation. But $R$ is not unimodular in our
sense.
\end{exam}

It is not difficult to see that any Jacobian Poisson structure over
$R=k[x_1, x_2, \cdots, x_n]$ is unimodular (see also \cite{Prz}).
The converse is true for polynomial Poisson algebras in three
variables.

\begin{prop} \cite[Theorem 5]{Prz} Any unimodular Poisson structure over $R=k[x_1, x_2, x_3]$ is a
Jacobian Poisson structure.
\end{prop}

\begin{proof} Suppose that $\{x_i, x_j\}=h_{ij} \in R \, (1 \leq i,  j \leq 3)$
and $(h_{ij})_{3 \times 3}$ is a skew-symmetric matrix. For any $f,
g \in R, \{f, g\}=\sum_{i, j =1}^3 \frac{\partial f}{\partial
x_i}\frac{\partial g}{\partial x_j}\{x_i, x_j\}$, i.e.,
$$
\begin{aligned}
\{f, g\}=&\frac{\partial f}{\partial x_1}\frac{\partial g}{\partial
x_2}h_{12}
+ \frac{\partial f}{\partial x_1}\frac{\partial g}{\partial x_3}h_{13}\\
-&\frac{\partial f}{\partial x_2}\frac{\partial g}{\partial
x_1}h_{12}
+\frac{\partial f}{\partial x_2}\frac{\partial g}{\partial x_3}h_{23}\\
- &\frac{\partial f}{\partial x_3}\frac{\partial g}{\partial
x_1}h_{13}
-\frac{\partial f}{\partial x_3}\frac{\partial g}{\partial x_2}h_{23}.\\
\end{aligned}
$$
$\{-,-\}$ is a Lie structure if and only if that
$$
\begin{aligned}
&\{\{x_1, x_2\}, x_3\}+\{\{x_2, x_3\}, x_1\}+\{\{x_3, x_1\}, x_2\}\\
=& \frac{\partial h_{12}}{\partial x_1}h_{13} + \frac{\partial
h_{12}}{\partial x_2}h_{23} - \frac{\partial h_{23}}{\partial
x_2}h_{12} - \frac{\partial h_{23}}{\partial x_3}h_{13} -
\frac{\partial h_{13}}{\partial x_1}h_{12} + \frac{\partial
h_{13}}{\partial x_3}h_{23}=0.\\
\end{aligned}
$$
The Poisson structure is unimodular if and only if
$$\frac{\partial h_{12}}{\partial x_2} + \frac{\partial h_{13}}{\partial
x_3}=0, - \frac{\partial h_{12}}{\partial x_1} + \frac{\partial
h_{23}}{\partial x_3}=0 \,\,\textrm{and} \, \, \frac{\partial
h_{13}}{\partial x_1} + \frac{\partial h_{23}}{\partial x_2}=0.
$$
Let $h_{12}=\frac{\partial u}{\partial x_3}$ for some $u \in R$. It
follows from $\frac{\partial h_{12}}{\partial x_2} + \frac{\partial
h_{13}}{\partial x_3}=0$ that $h_{13}=-\frac{\partial u}{\partial
x_2} + f(x_1, x_2)$ for some $f(x_1, x_2) \in k[x_1, x_2].$
Obviously, $h_{12}=\frac{\partial v}{\partial x_3}$ and
$h_{13}=-\frac{\partial v}{\partial x_2}$  for some $v \in R$. By
$\frac{\partial h_{13}}{\partial x_1} + \frac{\partial
h_{23}}{\partial x_2}=0$, $h_{23}=\frac{\partial v}{\partial x_1} +
g(x_1, x_3)$ for some $g(x_1, x_3) \in k[x_1, x_3].$ Since
$-\frac{\partial h_{12}}{\partial x_1} + \frac{\partial
h_{23}}{\partial x_3}=0$, $\frac{\partial g}{\partial x_3}=0$ and
$g(x_1, x_3) \in k[x_1]$. Hence, there is some $w \in R$ such that
$h_{12}=\frac{\partial w}{\partial x_3}, h_{13}=-\frac{\partial
w}{\partial x_2}$ and $h_{23}=\frac{\partial w}{\partial x_1}.$ It
follows that the Poisson structure on $R=k[x_1, x_2, x_3]$ is a
Jacobian Poisson structure given by $w \in R$.
\end{proof}

\subsection{Twisted Poisson module structure}

There is a Poincar\'{e} duality between the Poisson homology and
cohomology in the case of unimodular Poisson manifold \cite{Xu}.

For the polynomial Poisson algebra $R=k[x_1,x_2,\cdots,x_n]$ with  quadratic Poisson structure
 $\{x_i, x_j\}=a_{ij}x_i x_j$, where
$(a_{ij})\in M_n(k)$ is a skew-symmetric matrix, Launois-Richard
\cite[Theorem 3.4.2]{LR} proved a twisted Poincar\'{e} duality
between the Poisson cohomology $\PH^*(R,R)$ and the Poisson homology
$\PH_*(R,R_t)$, where $R_t$ is a twisted Poisson module structure of
$R$ given by the semi-classical limit.

Let $\mathfrak{g}$ be an $n$-dimensional $k$-Lie algebra with a
basis $\{x_1,x_2,\cdots,x_n\}$. Then the symmetric algebra
$S=S(\mathfrak{g}) \cong k[x_1,x_2,\cdots,x_n]$ is a Poisson algebra
with Poisson bracket defined by $\{x_i, x_j\}=[x_i, x_j]$. Zhu
\cite[Theorem 3.6]{Zhu} proved a twisted Poincar\'{e} duality
between the Poisson cohomology $\PH^*(S,S)$ and the Poisson homology
$\PH_*(S,S_t)$, where $S_t$ is a twisted Poisson module structure of
$S$ obtained from the Nakayama automorphism of the enveloping
algebra $U(\mathfrak{g}).$

The following proposition shows that any Poisson module can be twisted by a Poisson derivation.

\begin{prop} \label{twist-poisson-module-structure} Let $D \in \mathfrak{X}^1(R)$ be a Poisson derivation, and
$M$ be a right Poisson  $R$-module. Define a new bilinear map $\{-,
-\}_{M^D}: M \times R \to M$ as
 $$\{m, a\}_{M^D} : = \{m, a\}_M + m \cdot D(a).$$
Then the $R$-module $M$ with $\{-, -\}_{M^D}$ is a right Poisson
 $R$-module, which is called the {\bf twisted Poisson module}
of $M$ twisted by the Poisson derivation $D$, denoted by $M^D$.
\end{prop}

\begin{proof}For any $m \in M$ and $ a, b \in R$, since $ D (\{a,b\}) = \{
D(a),b\} + \{a,D(b)\},$
\begin{align*}& \{m, \{a, b\}\}_{M^D}
= \{m,\{a, b\}\} +  m \cdot D(\{a,b\}) \\
= & \{\{m, a\}, b\} - \{\{m, b\}, a\} +  m \cdot \{ D(a),b\} + m \cdot\{a,D(b)\}\\
= &\{\{m, a\}, b\} + \{m \cdot D(a), b\} - \{m, b\}  \cdot D(a)\\
& - \{\{m, b\}, a\} - \{m \cdot D(b), a\} + \{m, a\} \cdot D(b)  \\
= &\{\{m, a\}, b\} + \{m \cdot D(a), b\} + (\{m, a\} + m
\cdot D(a)) \cdot D(b)\\
& - \{\{m, b\}, a\} - \{m \cdot D(b), a\} -
(\{m, b\} + m \cdot D(b)) \cdot D(a)  \\
=&\{\{m, a\}_{M^D}, b\} + \{m, a\}_{M^D} \cdot D(b) - \{\{m, b\}_{M^D}, a\} - \{m, b\}_{M^D} \cdot D(a)\\
= & \{\{m, a\}_{M^D}, b\}_{M^D} - \{\{m, b\}_{M^D}, a\}_{M^D}.
\end{align*}
Hence $(M, \{-, -\}_{M^D})$ is a right Lie-module over the Lie
algebra $(R; \{-, -\})$, that is, condition (2) in Definition
\ref{poisson-module} holds.

Next, we claim that the conditions (3) and (4) in Definition
\ref{poisson-module} hold. 
\begin{align*} &\{ma, b\}_{M^D}
= \{ma, b\} +  ma \cdot D(b)
= \{m, b\}a + m\{a, b\}  +  ma \cdot D(b) \\
=& [\{m, b\} + m \cdot D(b)]a + m\{a, b\}  =\{m, b\}_{M^D} a + m\{a,
b\};
\end{align*}
and
\begin{align*}& \{m, ab\}_{M^D} =\{m, ab\} +  m \cdot D(ab)
=  \{m, a\}b + \{m, b\}a  +  m \cdot [D(a)b + a
D(b)]\\
= &[\{m, a\} + m \cdot D(a)]b + [\{m, b\} +  m \cdot D(b)]a  = \{m,
a\}_{M^D} b + \{m, b\}_{M^D} a.
\end{align*}
By definition, $M$ with $\{-, -\}_{M^D}$ is a Poisson right
$R$-module.
\end{proof}

%

\begin{exam}\label{Launois+zhu} (1) Consider the polynomial algebra $R=k[x_1,x_2,\cdots,x_n]$
with the quadratic Poisson structure given by $\{x_i,
x_j\}=a_{ij}x_i x_j$, where $(a_{ij})\in M_n(k)$ is a skew-symmetric
matrix \cite{LR}. By Lemma \ref{m-class-poly},
$\phi_\eta(x_i)=(\sum_{j=1}^n a_{ij})x_i$. So the Poisson module
$M=R_t$ given in \cite[\S 3.1(6)]{LR} is exactly the twisted Poisson
$R$-module of $R$ twisted by the modular derivation $\phi_{\eta}$ as
defined in Proposition \ref{twist-poisson-module-structure}.

(2) Let $\mathfrak{g}$ be an $n$-dimensional $k$-Lie algebra with a
basis $\{x_1,x_2,\cdots,x_n\}$. Then the symmetric algebra
$S=S(\mathfrak{g}) \cong k[x_1,x_2,\cdots,x_n]$ is a Poisson algebra
with Poisson bracket defined by $\{x_i, x_j\}=[x_i, x_j]$. One can
check that in this case the modular derivation
$\phi_\eta(x_i)=\textrm{tr}(\textrm{ad}\, x_i)$. The Poisson module $S_t$
given in \cite[\S 3]{Zhu} is exactly the twisted Poisson $S$-module
of $S$ twisted by the modular derivation $\phi_{\eta}$.
\end{exam}

The main result in this paper which will be proved in next section
is that there is a twisted Poincar\'{e} duality $\PH^*(R, M) \cong
\PH_{n-*}(R, M_t)$ for any Poisson module $M$ over polynomial
Poisson algebra $R=k[x_1,x_2,\cdots,x_n]$, where $M_t$ is the
twisted Poisson module of $M$ twisted by the modular derivation of
$R$. This generalizes the main results in \cite{LR} and \cite{Zhu}.
If the Poisson structure of $R$ is unimodular, then it reduces to
the Poincar\'{e} duality $\PH^*(R, M) \cong \PH_{n-*}(R, M)$ for any
Poisson $R$-module $M$.

The following example shows that Poincar\'{e} duality does not hold
between the Poisson homology and cohomology even for polynomial
Poisson algebras in three variables with Hamiltonian modular
derivation.

\begin{exam} Let $R=k[x, y, z]$ with the Poisson structure as given in Example \ref{non-unimodular-example}.
Consider the Poisson chain complex:
$$0 \to \Omega^3(R)\stackrel{\partial}{\longrightarrow} \Omega^2(R)
\stackrel{\partial}{\longrightarrow} \Omega^1(R)\stackrel{\partial}{\longrightarrow} \Omega^0(R) \to 0,$$ which is
$$0 \to R \de\! x \wedge \de\! y \wedge \de\! z \stackrel{\partial}{\longrightarrow}
R \de\! x \wedge \de\! y  \oplus R \de\! y\wedge \de\! z \oplus R \de\!
z\wedge \de\! x \stackrel{\partial}{\longrightarrow}$$
$ R \de\! x
\oplus R \de\! y \oplus R \de\! z \stackrel{\partial}{\longrightarrow} R
\to 0,$ with
\begin{align*}
\partial & (f\de\! x \wedge \de\! y \wedge \de\! z)=- \frac{\partial f}{\partial z}
\de\! y \wedge \de\! z + \frac{\partial f}{\partial z}y \de\! x \wedge \de\!
z + ( \frac{\partial f}{\partial x} + \frac{\partial f}{\partial y}y
+ f) \de\! x \wedge \de\! y.
\end{align*}
It follows that $\PH_3(R,R)= \ker \partial \cong \{ f \in R \mid  \frac{\partial f}{\partial z}=0, \frac{\partial f}{\partial x} + \frac{\partial f}{\partial y}y + f =0\}$.

Suppose that $f = b_0(x) + b_1(x) y + \cdots +  b_m(x) y^m$. Then  $\frac{\partial f}{\partial x} + \frac{\partial f}{\partial y}y + f =0$ if and only if $b_i(x)=-b'_i(x) -i b_i(x) \, (i \geq 1)$ and $b_0(x) = -b'_0(x)$ if and only if $f=0$.
Hence $\PH_3(R,R)= 0$.

Next we compute the Poisson homology $\PH_3(R,R_t).$
Consider the Poisson chain complex:
$$0 \to R_t \otimes  \Omega^3(R)\stackrel{\partial}{\longrightarrow}  R_t \otimes \Omega^2(R)\stackrel{\partial}{\longrightarrow} R_t \otimes \Omega^1(R)\stackrel{\partial}{\longrightarrow}  R_t \otimes \Omega^0(R) \to 0,$$ which is
$0 \to R \de\! x \wedge \de\! y \wedge \de\! z
\stackrel{\partial}{\longrightarrow} R \de\! x \wedge \de\! y  \oplus R
\de\! y\wedge \de\! z \oplus R \de\! z\wedge \de\! x
\stackrel{\partial}{\longrightarrow}$\\ $R \de\! x \oplus R \de\! y
\oplus R \de\! z \stackrel{\partial}{\longrightarrow} R \to 0,$ with
\begin{align*}
\partial & (g\de\! x \wedge \de\! y \wedge \de\! z)=- \frac{\partial g}{\partial z} \de\! y \wedge \de\! z
+ \frac{\partial g}{\partial z}y \de\! x \wedge \de\! z + (
\frac{\partial g}{\partial x} + \frac{\partial g}{\partial y}y) \de\!
x \wedge \de\! y.
\end{align*}
It follows that $\PH_3(R,R_t) = \ker \partial  \cong \{ g \in R \mid  \frac{\partial g}{\partial z}=0,
\frac{\partial g}{\partial x} + \frac{\partial g}{\partial y}y =0\}$.
Suppose that $g = c_0 + y f(x,y) + x h(x)$. Then, on one hand
$\frac{\partial g}{\partial x}=y\frac{\partial f}{\partial x} + h(x) + x \frac{\partial h}{\partial x} $,
and on the other hand, $\frac{\partial g}{\partial x}=-y\frac{\partial g}{\partial y}$.
It follows that  $h(x) + x \frac{\partial h}{\partial x} =0$ and $h(x)=0$. Hence $g = c_0 + y f(x,y)$. Then
$\frac{\partial g}{\partial x} + \frac{\partial g}{\partial y}y =0$
if and only if $\frac{\partial f}{\partial x} + \frac{\partial f}{\partial y}y + f =0$,
if and only if $f=0$, as showed before.
Hence $g=c_0$ and $\PH_3(R,R_t) \cong k$.

It follows from our main result Theorem \ref{main-theorem},  $\PH^0(R, R) \cong \PH_3(R,R_t) \cong k$.
Hence  $\PH^0(R, R)$ is not isomorphic to $\PH_3(R,R)$.
\end{exam}

\section {Poincar\'{e} duality of polynomial Poisson algebras}

In this section, let $R=k[x_1, x_2, \cdots, x_n]$ be a polynomial
Poisson algebra with Poisson bracket $\{-, -\},$ and $M$ be an
arbitrary  right Poisson  $R$-module. By using the twisted Poisson
$R$-module $M_t$ of $M$ twisted by the modular derivation
$\phi_{\eta}$, where $\eta=\de\!x_1 \wedge \de\!x_2  \wedge \cdots \wedge \de\!x_n$ is the
volume from of $R$, we prove a twisted Poincar\'{e} duality
between the Poisson cohomology of $R$ with values in $M$ and the
Poisson homology of $R$ with values in the twisted Poisson module
$M_t$. This generalizes the duality theorem for unimodular Poisson algebras and also the main
results in \cite{LR} for quadratic Poisson algebras and in
\cite{Zhu} for linear Poisson algebras.

For any $R$-module $M$, let $\ddag^p_M \colon  \mathfrak{X}^p(M) \to M \otimes_R
\Omega^{n-p}(R)$ be the map 
\begin{equation}\label{definition-of-ddag}
F \mapsto \sum_{\sigma \in S_{p, n-p}} \sgn
(\sigma) F(x_{\sigma(1)}, \cdots, x_{\sigma(p)}) \otimes
\de\!x_{\sigma(p+1)} \wedge \cdots \wedge \de\!x_{\sigma(n)}.
\end{equation}

\begin{lem} \label{iso-chi-omega} For any $R$-module $M$, $\ddag^p_M \colon  \mathfrak{X}^p(M) \to M \otimes_R
\Omega^{n-p}(R)$ is an isomorphism.
\end{lem}

\begin{proof}  It is well known that $\ddag^p_R:
\mathfrak{X}^p(R) \cong \Omega^{n-p}(R), F \mapsto \iota_F\eta$ is an isomorphism, with the inverse
$\Omega^{n-p}(R) \to \mathfrak{X}^p(R), Q \mapsto \iota_Q \eta^*$,
where $\eta^*=\partial/ {\partial x_1} \wedge  \cdots \wedge \partial/ {\partial x_n}
\in \mathfrak{X}^n(R) \cong \Hom_R(\Omega^n(R) ,R)$.

Since $\Omega^{p}(R)$ is a finitely generated projective $R$-module
($p \geq 0$), with a dual basis
$$\{\de\!x_{i_1} \wedge \cdots \wedge \de\!x_{i_p},
\partial/ {\partial x_{i_1}} \wedge \cdots \wedge \partial/ {\partial x_{i_p}}
\}_{1 \leq i_1 < \cdots < i_p \leq n},  $$
 there are canonical isomorphisms of left $R$-modules
\begin{align*}
\mathfrak{X}^p(M) & \cong \Hom_R(\Omega^p(R), M)\\
 & \cong    M \otimes_R \Hom_R(\Omega^p(R), R)\\
 & \cong M \otimes_R  \mathfrak{X}^p(R) \\
 & \cong  M \otimes_R \Omega^{n-p}(R),
\end{align*}
where the last isomorphism is $\id_M \otimes \ddag^p_R$.
The isomorphism  $ \mathfrak{X}^p(M) \to M \otimes_R
\Omega^{n-p}(R)$ given by the composition
$$ \mathfrak{X}^p(M) \to  M \otimes_R \Hom_R(\Omega^p(R), R)   \to M
\otimes_R \Omega^{n-p}(R),$$ is
\begin{align*} F & \mapsto \sum_{1 \leq i_1 < \cdots < i_p \leq n} F(x_{i_1}, \cdots, x_{i_p}) \otimes
\partial/ {\partial x_{i_1}} \wedge \cdots \wedge \partial/ {\partial x_{i_p}}\\
&\mapsto \sum_{ i_1 < \cdots < i_p } \sum_{\sigma \in S_{p, n-p}} \sgn
(\sigma) F(x_{i_1}, \cdots, x_{i_p}) (\partial/ {\partial x_{i_1}}
\wedge \cdots \wedge \partial/ {\partial x_{i_p}})
(x_{\sigma(1)}, \cdots, x_{\sigma(p)}) \otimes \\
& \quad \quad \de\!x_{\sigma(p+1)} \wedge \cdots \wedge \de\!x_{\sigma(n)}\\
& =  \sum_{\sigma \in S_{p, n-p}} \sgn (\sigma) F(x_{\sigma(1)},
\cdots, x_{\sigma(p)}) \otimes  \de\!x_{\sigma(p+1)} \wedge \cdots
\wedge \de\!x_{\sigma(n)}
\end{align*}
which is exactly $\ddag^p_M$.
\end{proof}

In the following, let $\dag^p_M=(-1)^{\frac{p(p+1)}{2}}\ddag^p_M$.

\begin{prop}\label{commu-dia}
Let $R=k[x_1, x_2, \cdots, x_n]$ be a polynomial Poisson algebra,
$M$ be a right Poisson $R$-module and $M_t$ be the twisted Poisson module of
 $M$ twisted by the modular derivation $\phi_{\eta}$.
Then the following diagram is commutative.
\begin{equation} \label{maindiagram}
\xymatrix{
   \mathfrak{X}^p(M) \ar[d]_{\dag^p_M} \ar[r]^{{\delta}^p}
        &   \mathfrak{X}^{p+1}(M)  \ar[d]_{\dag^{p+1}_M}    \\
  M_t \otimes_R \Omega^{n-p}(R)    \ar[r]^{{\partial}_{n-p}}
        & M_t \otimes_R \Omega^{n-p-1}(R).
     }
\end{equation}
\end{prop}

\begin{proof} Suppose $f \in   \mathfrak{X}^p(M)$. Then, by \eqref{poison-cochain-complex}
 and \eqref{definition-of-ddag},
\begin{align*}
 &(\dag^{p+1}_M \cdot \delta^p) (f)= \dag^{p+1}_M (\delta^p (f))\\
 =&(-1)^{\frac{(p+1)(p+2)}{2}}\sum_{\substack{1 \leq i \leq p+1\\\sigma \in S_{p+1, n-p-1} }} \sgn
(\sigma)(-1)^{i}\{f(x_{\sigma(1)}, \cdots, \widehat{x_{\sigma(i)}}
\cdots, x_{\sigma(p+1)}), x_{\sigma(i)}\} \otimes \\
& \quad \de\!x_{\sigma(p+2)} \wedge \cdots \wedge \de\!x_{\sigma(n)} \\
& + (-1)^{\frac{(p+1)(p+2)}{2}} \sum_{\substack{1
\leq i < j \leq p+1\\\sigma \in S_{p+1, n-p-1}}} \sgn
(\sigma)(-1)^{i+j} \\
& f(\{x_{\sigma(i)}, x_{\sigma(j)}\}, x_{\sigma(1)}, \cdots,
\widehat{x_{\sigma(i)}}, \cdots, \widehat{x_{\sigma(j)}}, \cdots,
x_{\sigma(p+1)}) \otimes \de\!x_{\sigma(p+2)} \wedge \cdots \wedge
\de\!x_{\sigma(n)}.
\end{align*}
Denote the first term  by $U$ and the second one by $V$.

Let $S_{p, 1, n-p-1}$ be the set of all $(p,1,n-p-1)$-shuffles,
which are all the permutations $\sigma \in S_n$ such that
$\sigma(1)< \cdots < \sigma(p)$ and $\sigma(p+2)< \cdots <
\sigma(n)$. For any $\sigma \in S_{p+1, n-p-1}$ and any $i \in \{1,
2, \cdots, p+1\}$, if we move $\sigma(i)$ backward to the position
$p+1$, then we get a shuffle in $S_{p, 1, n-p-1}$. This gives the following one
to one correspondence
$$S_{p+1, n-p-1} \times \{1, 2, \cdots, p+1\} \to S_{p, 1, n-p-1},
(\sigma, i) \mapsto \tau = \sigma \alpha_i,$$ where $\alpha_i=(i,
i+1, \cdots, p+1)$.
Note that $\sgn (\alpha_i) =(-1)^{p-i+1}$. Then
\begin{align*}
 U &=(-1)^{\frac{(p+1)(p+2)}{2}} \sum_{\substack{1 \leq i \leq
p+1\\\sigma \in S_{p+1, n-p-1}}}(-1)^{p+1}
\sgn (\sigma \alpha_i)\\
&\quad \{f(x_{\sigma\alpha_i(1)}, \cdots, x_{\sigma\alpha_i(i-1)},
x_{\sigma\alpha_i(i)}, \cdots, x_{\sigma\alpha_i(p)}),
x_{\sigma\alpha_i(p+1)}\} \otimes\\
&\quad \de\!x_{\sigma\alpha_i(p+2)} \wedge \cdots \wedge \de\!x_{\sigma\alpha_i(n)}\\
&=(-1)^{\frac{p(p+1)}{2}}\sum_{\tau\in S_{p,1,
n-p-1}}\sgn(\tau)\{f(x_{\tau(1)},\cdots,x_{\tau(p}),
x_{\tau(p+1)}\}\otimes \\
& \quad \de\!x_{\tau(p+2)}\wedge\cdots\wedge \de\!x_{\tau(n)}.
\end{align*}

Since $f$ is a derivation
in each argument, $f(X, y_2, \cdots, y_n) = \sum_{l=1}^n \frac{\partial
{X}}{\partial x_l} f(x_l, y_2, \cdots, y_n)$. So,
\begin{align*}
V=&(-1)^{\frac{(p+1)(p+2)}{2}}\sum_{\substack{1
\leq i < j \leq p+1\\\sigma \in S_{p+1, n-p-1}}} \sgn (\sigma)(-1)^{i+j} \sum_{l=1}^n
\frac{\partial
{\{x_{\sigma(i)}, x_{\sigma(j)}\}}}{\partial x_l} \\
& \quad f(x_l, x_{\sigma(1)}, \cdots, \widehat{x_{\sigma(i)}},
\cdots, \widehat{x_{\sigma(j)}}, \cdots, x_{\sigma(p+1)}) \otimes
\de\!x_{\sigma(p+2)} \wedge \cdots \wedge \de\!x_{\sigma(n)}\\
=&(-1)^{\frac{(p+1)(p+2)}{2}}\sum_{\substack{1
\leq i < j \leq p+1\\\sigma \in S_{p+1, n-p-1}}} \sgn (\sigma)(-1)^{i+j} \sum_{l=1}^n
\frac{\partial
{\{x_{\sigma(i)}, x_{\sigma(j)}\}}}{\partial x_{\sigma (l)}} \\
& \quad  f(x_{\sigma(l)}, x_{\sigma(1)}, \cdots,
\widehat{x_{\sigma(i)}}, \cdots, \widehat{x_{\sigma(j)}}, \cdots,
x_{\sigma(p+1)}) \otimes
\de\!x_{\sigma(p+2)} \wedge \cdots \wedge \de\!x_{\sigma(n)}\\
=&(-1)^{\frac{(p+1)(p+2)}{2}}\sum_{\substack{1
\leq i < j \leq p+1\\\sigma \in S_{p+1, n-p-1}}} \sgn (\sigma)(-1)^{j-1} \frac{\partial
{\{x_{\sigma(i)}, x_{\sigma(j)}\}}}{\partial x_{\sigma (i)}} \\
& \quad  f(x_{\sigma(1)}, \cdots, x_{\sigma(i)}, \cdots,
\widehat{x_{\sigma(j)}}, \cdots, x_{\sigma(p+1)}) \otimes
\de\!x_{\sigma(p+2)} \wedge \cdots \wedge \de\!x_{\sigma(n)}\\
&+ (-1)^{\frac{(p+1)(p+2)}{2}}\sum_{\substack{1
\leq i < j \leq p+1\\\sigma \in S_{p+1, n-p-1}}} \sgn (\sigma)(-1)^{i} \frac{\partial
{\{x_{\sigma(i)}, x_{\sigma(j)}\}}}{\partial x_{\sigma (j)}} \\
& \quad f(x_{\sigma(1)}, \cdots,  \widehat{x_{\sigma(i)}}, \cdots,
x_{\sigma(j)}, \cdots, x_{\sigma(p+1)}) \otimes
\de\!x_{\sigma(p+2)} \wedge \cdots \wedge \de\!x_{\sigma(n)}\\
&+ (-1)^{\frac{(p+1)(p+2)}{2}}\sum_{\substack{1
\leq i < j \leq p+1\\\sigma \in S_{p+1, n-p-1}}} \sgn (\sigma)(-1)^{i+j} \sum_{l=p+2}^n
\frac{\partial
{\{x_{\sigma(i)}, x_{\sigma(j)}\}}}{\partial x_{\sigma (l)}} \\
& \quad f(x_{\sigma(l)}, x_{\sigma(1)}, \cdots,
\widehat{x_{\sigma(i)}}, \cdots, \widehat{x_{\sigma(j)}}, \cdots,
x_{\sigma(p+1)}) \otimes
\de\!x_{\sigma(p+2)} \wedge \cdots \wedge \de\!x_{\sigma(n)}\\
=&(-1)^{\frac{(p+1)(p+2)}{2}}\sum_{\sigma  \in S_{p+1, n-p-1}} \sgn
(\sigma) \sum_{j=2}^{p+1} \sum_{i=1}^{j-1} (-1)^{j-1} \frac{\partial
{\{x_{\sigma(i)}, x_{\sigma(j)}\}}}{\partial x_{\sigma (i)}} \\
& \quad  f(x_{\sigma(1)}, \cdots, x_{\sigma(i)}, \cdots,
\widehat{x_{\sigma(j)}}, \cdots, x_{\sigma(p+1)}) \otimes
\de\!x_{\sigma(p+2)} \wedge \cdots \wedge \de\!x_{\sigma(n)}\\
&+(-1)^{\frac{(p+1)(p+2)}{2}}\sum_{\sigma \in S_{p+1, n-p-1}} \sgn
(\sigma) \sum_{i=1}^p \sum_{j=i+1}^{p+1} (-1)^{i} \frac{\partial
{\{x_{\sigma(i)}, x_{\sigma(j)}\}}}{\partial x_{\sigma (j)}} \\
& \quad f(x_{\sigma(1)}, \cdots,  \widehat{x_{\sigma(i)}}, \cdots,
x_{\sigma(j)}, \cdots, x_{\sigma(p+1)}) \otimes
\de\!x_{\sigma(p+2)} \wedge \cdots \wedge \de\!x_{\sigma(n)}\\
&+(-1)^{\frac{(p+1)(p+2)}{2}}\sum_{\substack{1
\leq i < j \leq p+1\\\sigma \in S_{p+1, n-p-1}}} \sgn (\sigma)(-1)^{i+j} \sum_{l=p+2}^n
\frac{\partial
{\{x_{\sigma(i)}, x_{\sigma(j)}\}}}{\partial x_{\sigma (l)}} \\
& \quad  f(x_{\sigma(l)}, x_{\sigma(1)}, \cdots,
\widehat{x_{\sigma(i)}}, \cdots, \widehat{x_{\sigma(j)}}, \cdots,
x_{\sigma(p+1)}) \otimes
\de\!x_{\sigma(p+2)} \wedge \cdots \wedge \de\!x_{\sigma(n)}\\
=&(-1)^{\frac{(p+1)(p+2)}{2}}\sum_{\sigma  \in S_{p+1, n-p-1}} \sgn
(\sigma) \sum_{i=1}^{p+1} (-1)^{i} \sum_{j=1}^{p+1}\frac{\partial
{\{x_{\sigma(i)}, x_{\sigma(j)}\}}}{\partial x_{\sigma (j)}} \\
&\quad f(x_{\sigma(1)}, \cdots, \widehat{x_{\sigma(i)}}, \cdots,
x_{\sigma(p+1)}) \otimes \de\!x_{\sigma(p+2)} \wedge \cdots \wedge
\de\!x_{\sigma(n)}\\
&+(-1)^{\frac{(p+1)(p+2)}{2}}\sum_{\substack{1
\leq i < j \leq p+1\\\sigma \in S_{p+1, n-p-1}}} \sgn (\sigma)(-1)^{i+j} \sum_{l=p+2}^n
\frac{\partial
{\{x_{\sigma(i)}, x_{\sigma(j)}\}}}{\partial x_{\sigma (l)}} \\
& \quad f(x_{\sigma(l)}, x_{\sigma(1)}, \cdots,
\widehat{x_{\sigma(i)}}, \cdots, \widehat{x_{\sigma(j)}}, \cdots,
x_{\sigma(p+1)}) \otimes
\de\!x_{\sigma(p+2)} \wedge \cdots \wedge \de\!x_{\sigma(n)}\\
=&V_1 + V_2,
\end{align*}
where $V_1$ is the first term, and $V_2$ is the second one of the
last expression.

As we did for $U$,
\begin{align*}
V_1=&(-1)^{\frac{(p+1)(p+2)}{2}}\sum_{\sigma \in S_{p+1, n-p-1}}
\sgn (\sigma) \sum_{i=1}^{p+1} (-1)^{i}
\sum_{j=1}^{p+1}\frac{\partial
{\{x_{\sigma(i)}, x_{\sigma(j)}\}}}{\partial x_{\sigma (j)}} \\
&\quad f(x_{\sigma(1)}, \cdots, \widehat{x_{\sigma(i)}}, \cdots,
x_{\sigma(p+1)}) \otimes \de\!x_{\sigma(p+2)} \wedge \cdots \wedge
\de\!x_{\sigma(n)}\\
&\textrm {can be rewritten as}\\
V_1 =&(-1)^{\frac{p(p+1)}{2}}\sum_{\tau \in S_{p, 1, n-p-1}} \sgn
(\tau)
 \sum_{j=1}^{p+1}\frac{\partial
{\{x_{\tau (p+1)}, x_{\tau(j)}\}}}{\partial x_{\tau (j)}} \\
& \quad f(x_{\tau(1)}, \cdots, x_{\tau(p)}) \otimes \de\!x_{\tau(p+2)}
\wedge
\cdots \wedge \de\!x_{\tau(n)}.\\
\end{align*}

On the other hand, by using \eqref{poison-chain-complex},
\eqref{definition-of-ddag} and $\de\!\{x_{\sigma(i)},
x_{\sigma(j)}\}=\sum_{l=1}^n \frac{\partial {\{x_{\sigma(i)},
x_{\sigma(j)}\}}}{\partial x_l} \de\!x_l$,
\begin{align*}
&(\partial_{n-p} \cdot \dag^p_M) (f)= \partial_{n-p} (\dag^p_M (f))\\
 =&(-1)^{\frac{p(p+1)}{2}}\sum_{\sigma \in S_{p, n-p}} \sgn
(\sigma) \sum_{i=1}^{n-p}(-1)^{i-1}\{f(x_{\sigma(1)}, \cdots,
 x_{\sigma(p)}), x_{\sigma(p+i)}\}_{M_t}
\otimes\\
& \quad \de\!x_{\sigma(p+1)} \wedge \cdots \widehat{\de\!x_{\sigma(p+i)}} \cdots \wedge \de\!x_{\sigma(n)}  \\
& + (-1)^{\frac{p(p+1)}{2}}\sum_{\substack{1 \leq i<j
\leq n-p\\\sigma \in S_{p, n-p}}} \sgn
(\sigma)(-1)^{i+j}  f(x_{\sigma(1)}, \cdots, x_{\sigma(p)}) \otimes\\
& \quad \de\!\{x_{\sigma(p+i)}, x_{\sigma(p+j)}\} \wedge
\de\!x_{\sigma(p+1)} \wedge
 \cdots  \widehat{\de\!x_{\sigma(p+i)}} \cdots
\widehat{\de\!x_{\sigma(p+j)}}  \cdots \wedge \de\!x_{\sigma(n)}\\
=&L + N,
\end{align*}
where $L$ is the first term, and $N$ is the second one of the last
expression.

Consider the one-to-one correspondence $$S_{p, n-p} \times \{1, 2,
\cdots, n-p\} \to S_{p, 1, n-p-1}, (\sigma, i) \mapsto \tau=\sigma
\beta_i,$$ where $\beta_i=(p+i, p+i-1, \cdots, p+1)$. Note that
$\sgn (\beta_i) =(-1)^{i-1}$. Then
\begin{align*}
L=&(-1)^{\frac{p(p+1)}{2}} \sum_{\sigma \in S_{p, n-p}} \sgn
(\sigma) \sum_{i=1}^{n-p}(-1)^{i-1}\{f(x_{\sigma(1)}, \cdots,
 x_{\sigma(p)}), x_{\sigma(p+i)}\}_{M_t}
\otimes\\
& \quad \de\!x_{\sigma(p+1)} \wedge \cdots  \widehat{\de\!x_{\sigma(p+i)}} \cdots \wedge \de\!x_{\sigma(n)} \\
=&(-1)^{\frac{p(p+1)}{2}} \sum_{\sigma \in S_{p, n-p}}
\sum_{i=1}^{n-p} \sgn (\sigma \beta_i) \{f(x_{\sigma \beta_i(1)},
\cdots, x_{\sigma\beta_i(p)}), x_{\sigma\beta_i(p+1)}\}_{M_t}
\otimes\\
& \quad \de\!x_{\sigma \beta_i(p+2)} \wedge \cdots \wedge
\de\!x_{\sigma\beta_i(p+i)} \wedge
\de\!x_{\sigma\beta_i(p+i+1)}\wedge \cdots \wedge \de\!x_{\sigma \beta_i(n)} \\
=&(-1)^{\frac{p(p+1)}{2}} \sum_{\tau \in S_{p, 1, n-p-1}} \sgn(\tau)
\{f(x_{\tau(1)},
\cdots, x_{\tau(p)}), x_{\tau(p+1)}\}_{M_t} \otimes \de\!x_{\tau(p+2)} \wedge \cdots  \wedge \de\!x_{\tau(n)}  \\
=&(-1)^{\frac{p(p+1)}{2}} \sum_{\tau \in S_{p, 1, n-p-1}} \sgn(\tau)
\{f(x_{\tau(1)},
\cdots, x_{\tau(p)}), x_{\tau(p+1)}\} \otimes \de\!x_{\tau(p+2)} \wedge \cdots  \wedge \de\!x_{\tau(n)}  \\
&+(-1)^{\frac{p(p+1)}{2}} \sum_{\tau \in S_{p,1, n-p-1}} \sgn
(\tau)  f(x_{\tau(1)}, \cdots,
 x_{\tau(p)}) \cdot \phi_{\eta} (x_{\tau(p+1)})
\otimes\\
& \quad \de\!x_{\tau(p+2)} \wedge \cdots \wedge \de\!x_{\tau(n)} \\
=&L_1 + L_2,
\end{align*}
where $L_1$ is the first term, and $L_2$ is the second one of the
last expression.
\begin{align*}
N=& (-1)^{\frac{p(p+1)}{2}}\sum_{\sigma \in S_{p, n-p}}\sgn (\sigma)
\sum_{1 \leq i < j \leq n-p} (-1)^{i+j} f(x_{\sigma(1)}, \cdots,
x_{\sigma(p)}) \otimes \sum_{l=1}^n \\
& \frac{\partial {\{x_{\sigma(p+i)}, x_{\sigma(p+j)}\}}}{\partial
x_l} \,  \de\!x_l \wedge \de\!x_{\sigma(p+1)} \wedge
 \cdots  \widehat{\de\!x_{\sigma(p+i)}}  \cdots
\widehat{\de\!x_{\sigma(p+j)}} \cdots \wedge \de\!x_{\sigma(n)}\\
 =& (-1)^{\frac{p(p+1)}{2}} \sum_{\sigma \in S_{p, n-p}}\sgn (\sigma)
\sum_{1 \leq i < j \leq n-p} (-1)^{i+j} f(x_{\sigma(1)}, \cdots,
x_{\sigma(p)}) \otimes \sum_{l=1}^n \\
&  \frac{\partial {\{x_{\sigma(p+i)}, x_{\sigma(p+j)}\}}}{\partial
x_{\sigma(l)}}\,
 \de\!x_{\sigma(l)} \wedge \de\!x_{\sigma(p+1)} \wedge
 \cdots  \widehat{\de\!x_{\sigma(p+i)}}  \cdots
\widehat{\de\!x_{\sigma(p+j)}} \cdots \wedge \de\!x_{\sigma(n)}\\
=&(-1)^{\frac{p(p+1)}{2}} \sum_{\sigma \in S_{p, n-p}}\sgn (\sigma)
\sum_{1 \leq i < j \leq n-p} (-1)^{j-1} f(x_{\sigma(1)}, \cdots,
x_{\sigma(p)}) \otimes \\
&  \frac{\partial {\{x_{\sigma(p+i)}, x_{\sigma(p+j)}\}}}{\partial
x_{\sigma(p+i)}} \, \de\!x_{\sigma(p+1)} \wedge \cdots
\widehat{\de\!x_{\sigma(p+j)}} \cdots \wedge \de\!x_{\sigma(n)} \\
&  + (-1)^{\frac{p(p+1)}{2}} \sum_{\sigma \in S_{p, n-p}}\sgn
(\sigma) \sum_{1 \leq i < j \leq n-p} (-1)^{i} f(x_{\sigma(1)},
\cdots, x_{\sigma(p)}) \otimes \\
& \frac{\partial {\{x_{\sigma(p+i)}, x_{\sigma(p+j)}\}}}{\partial
x_{\sigma(p+j)}} \, \de\!x_{\sigma(p+1)}
\wedge \cdots \widehat{\de\!x_{\sigma(p+i)}}  \cdots \wedge \de\!x_{\sigma(n)}\\
&+ (-1)^{\frac{p(p+1)}{2}} \sum_{\sigma \in S_{p, n-p}}\sgn (\sigma)
\sum_{1 \leq i < j \leq n-p} (-1)^{i+j} f(x_{\sigma(1)}, \cdots,
x_{\sigma(p)}) \otimes \sum_{l=1}^p \\
&   \frac{\partial {\{x_{\sigma(p+i)}, x_{\sigma(p+j)}\}}}{\partial
x_{\sigma(l)}} \,
 \de\!x_{\sigma(l)} \wedge \de\!x_{\sigma(p+1)} \wedge
 \cdots \widehat{\de\!x_{\sigma(p+i)}} \cdots
\widehat{\de\!x_{\sigma(p+j)}} \cdots \wedge \de\!x_{\sigma(n)}\\
=&  (-1)^{\frac{p(p+1)}{2}} \sum_{\sigma \in S_{p, n-p}}\sgn
(\sigma) \sum_{i=1}^{n-p} \sum_{j=1}^{n-p} (-1)^i f(x_{\sigma(1)},
\cdots, x_{\sigma(p)}) \otimes \frac{\partial
{\{x_{\sigma(p+i)}, x_{\sigma(p+j)}\}}}{\partial x_{\sigma(p+j)}} \\
&  \de\!x_{\sigma(p+1)} \wedge \cdots
\widehat{\de\!x_{\sigma(p+i)}} \cdots \wedge \de\!x_{\sigma(n)} \\
&+ (-1)^{\frac{p(p+1)}{2}} \sum_{\sigma \in S_{p, n-p}}\sgn (\sigma)
\sum_{1 \leq i < j \leq n-p} (-1)^{i+j} f(x_{\sigma(1)}, \cdots,
x_{\sigma(p)}) \otimes \sum_{l=1}^p \\
& \frac{\partial {\{x_{\sigma(p+i)}, x_{\sigma(p+j)}\}}}{\partial
x_{\sigma(l)}} \,
  \de\!x_{\sigma(l)} \wedge \de\!x_{\sigma(p+1)} \wedge
 \cdots \widehat{\de\!x_{\sigma(p+i)}} \cdots
\widehat{\de\!x_{\sigma(p+j)}} \cdots \wedge \de\!x_{\sigma(n)}\\
=& N_1 + N_2,
\end{align*}
where $N_1$ is the first term, and $N_2$ is the second term of the
last expression.

As we did for $L$,
\begin{align*}
N_1&= (-1)^{\frac{p(p+1)}{2}} \sum_{\sigma \in S_{p, n-p}}\sgn
(\sigma) \sum_{i=1}^{n-p} \sum_{j=1}^{n-p} (-1)^i f(x_{\sigma(1)},
\cdots, x_{\sigma(p)}) \otimes \\
& \quad \frac{\partial {\{x_{\sigma(p+i)},
x_{\sigma(p+j)}\}}}{\partial x_{\sigma(p+j)}} \de\!x_{\sigma(p+1)}
\wedge \cdots \widehat{\de\!x_{\sigma(p+i)}} \cdots \wedge \de\!x_{\sigma(n)} \\
=& (-1)^{\frac{p(p+1)}{2}}\sum_{\tau \in S_{p,1, n-p-1}}\sgn (\tau)
\sum_{j=1}^{n-p} (-1) f(x_{\tau(1)}, \cdots, x_{\tau(p)})
\otimes \\
& \quad \frac{\partial {\{x_{\tau(p+1)}, x_{\tau(p+j)}\}}}{\partial
x_{\tau(p+j)}} \de\!x_{\tau(p+2)} \wedge \cdots \wedge \de\!x_{\tau(n)}.
\end{align*}

Obviously, $U=L_1$. By the following Lemma \ref{V1=L2+N1} and Lemma
\ref{V2=N2}, $V_1 = L_2 + N_1$ and $V_2 =  N_2$. It follows that
\begin{align*}
(\dag^{p+1}_M \cdot \delta_p) (f)&= U +V = U+ V_1 + V_2\\
 &=  L_1 + L_2 + N_1 +
N_2 = L + N \\
&= (\partial_{n-p} \cdot \dag^p_M) (f).
\end{align*}
So, the diagram \ref{maindiagram} is commutative.
\end{proof}

\begin{lem}\label{V1=L2+N1}
 $V_1= L_2 + N_1$.
\end{lem}
\begin{proof} By lemma \ref{m-class-poly},
\begin{align*}
&V_1 - N_1\\
 =&(-1)^{\frac{p(p+1)}{2}}\sum_{\tau \in S_{p, 1, n-p-1}}
\sgn (\tau) f(x_{\tau(1)}, \cdots, x_{\tau(p)}) \cdot\\
& (\sum_{j=1}^{p+1}\frac{\partial {\{x_{\tau (p+1)},
x_{\tau(j)}\}}}{\partial x_{\tau (j)}} +  \sum_{j=1}^{n-p}
\frac{\partial {\{x_{\tau(p+1)}, x_{\tau(p+j)}\}}}{\partial
x_{\tau(p+j)}})
 \otimes
\de\!x_{\tau(p+2)} \wedge \cdots \wedge \de\!x_{\tau(n)}\\
=&(-1)^{\frac{p(p+1)}{2}}\sum_{\tau \in S_{p, 1, n-p-1}} \sgn (\tau)
f(x_{\tau(1)}, \cdots, x_{\tau(p)}) \cdot \sum_{j=1}^n
\frac{\partial {\{x_{\tau (p+1)},
x_{\tau(j)}\}}}{\partial x_{\tau (j)}}\\
 &\quad  \otimes
\de\!x_{\tau(p+2)} \wedge \cdots \wedge \de\!x_{\tau(n)}\\
=&(-1)^{\frac{p(p+1)}{2}} \sum_{\tau \in S_{p,1, n-p-1}} \sgn
(\tau)  f(x_{\tau(1)}, \cdots,
 x_{\tau(p)}) \cdot \phi_{\eta} (x_{\tau(p+1)})\\
&\quad \otimes
\de\!x_{\tau(p+2)} \wedge \cdots \wedge \de\!x_{\tau(n)}\\
 =&L_2.
\end{align*}
%
%
%
%
\end{proof}

\begin{lem}\label{V2=N2}
 $V_2=N_2$.
\end{lem}

\begin{proof}Let $S_{p-1, 2, 1, n-p-2}$ denote the set of all $(p-1, 2, 1, n-p-2)$-shuffles,
which are the permutations $\sigma \in S_n$ such that
$\sigma(1)< \cdots < \sigma(p-1)$, $\sigma(p)< \sigma(p+1)$ and
$\sigma(p+3)< \cdots < \sigma(n)$. For any shuffle $\sigma \in
S_{p+1, n-p-1}$, any $i < j \in \{1, 2, \cdots, p+1\}$ and any $l
\in \{p+2, p+3, \cdots, n\}$, we can get a shuffle $\tau \in S_{p-1,
2, 1, n-p-2}$ by the following three steps. First move  $\sigma(l)$
backward to the position $p+2$ which can be realized as $\sigma
\beta_l$ with $\beta_l=(l, l-1, \cdots, p+2)$; then move $\sigma(j)$
forward to the position $p+1$ which is realized as $ \sigma \beta_l
\alpha_j'$ with $\alpha_j'=(j, j+1, \cdots, p+1)$; and finally move
$\sigma(i)$ forward to the position $p$ which is realized as $\sigma
\beta_l \alpha_j' \alpha_i$ with $\alpha_i=(i, i+1, \cdots, p)$. So
we have a map
$$S_{p+1, n-p-1} \times \{1\leq i < j \leq
 p+1\} \times \{p+2 \leq l \leq n \} \to S_{p-1, 2, 1, n-p-2},$$
 $$ (\sigma, i < j, l) \mapsto \tau = \sigma  \beta_l \alpha_j' \alpha_i,$$
This map turns out to be a bijection. Note that $\sgn (\beta_l)
=(-1)^{l-p}$, $\sgn (\alpha_j') =(-1)^{p+1-j}$, $\sgn (\alpha_i)
=(-1)^{p-i}$.

 Then
\begin{align*}
 V_2=& (-1)^{\frac{(p+1)(p+2)}{2}}\sum_{\substack{1 \leq i < j \leq p+1\\\sigma \in S_{p+1,
n-p-1}}} \sgn (\sigma)(-1)^{i+j}
\sum_{l=p+2}^n \frac{\partial
{\{x_{\sigma(i)}, x_{\sigma(j)}\}}}{\partial x_{\sigma (l)}} \\
& \quad f(x_{\sigma(l)}, x_{\sigma(1)}, \cdots,
\widehat{x_{\sigma(i)}}, \cdots, \widehat{x_{\sigma(j)}}, \cdots,
x_{\sigma(p+1)}) \otimes \de\!x_{\sigma(p+2)} \wedge \cdots \wedge
\de\!x_{\sigma(n)}\\
=& (-1)^{\frac{(p+1)(p+2)}{2}}\sum_{\substack{1 \leq i < j \leq
p+1\\\sigma \in S_{p+1, n-p-1}}} \sum_{l=p+2}^n  \sgn (\sigma
\beta_l \alpha_j' \alpha_i)(-1)^{p+l+1}\\
&\frac{\partial {\{x_{\sigma \beta_l \alpha_j' \alpha_i (p)},
x_{\sigma \beta_l\alpha_j' \alpha_i (p+1)}\}}} {\partial x_{\sigma
\beta_l \alpha_j' \alpha_i (p+2)}} f(x_{\sigma \beta_l\alpha_j'
\alpha_i (p+2)}, x_{\sigma \beta_l\alpha_j '\alpha_i(1)}, \cdots,
x_{\sigma \beta_l\alpha_j'
\alpha_i(p-1)}) \\
&\otimes (-1)^{l-p}\de\!x_{\sigma \beta_l\alpha_j' \alpha_i(p+2)} \wedge
\cdots
\wedge \de\!x_{\sigma \beta_l\alpha_j' \alpha_i (n)} \\
=&(-1)^{\frac{p(p+1)}{2}+p}\sum_{\tau \in S_{p-1,2, 1, n-p-2}} \sgn
(\tau) \frac{\partial
{\{x_{\tau(p)}, x_{\tau(p+1)}\}}}{\partial x_{\tau (p+2)}} \\
& \quad f(x_{\tau(p+2)}, x_{\tau (1)}, \cdots, x_{\tau (p-1)})
\otimes \de\!x_{\tau (p+2)} \wedge \cdots \wedge \de\!x_{\tau (n)}.
\end{align*}

On the other hand, for any shuffle $\sigma \in S_{p, n-p}$, any $i <
j \in \{1, 2, \cdots, n-p\}$ and any $l \in \{1, 2, \cdots, p\}$, we
can get a shuffle $\tau \in S_{p-1, 2, 1, n-p-2}$ in the following
three steps. First move $\sigma(p+i)$ backward to the position
$p+1$, which is realized as  $\sigma \beta_{p+i}$ with
$\beta_{p+i}=(p+i, p+i-1, \cdots, p+1)$; then move $\sigma(p+j)$
backward to the position $p+2$, which is realized as  $\sigma
\beta_{p+i}\beta_{p+j}'$ with $\beta_{p+j}'=(p+j, p+j-1, \cdots,
p+2)$; and finally move $\sigma(l)$ forward to the position $p+2$,
which is realized as  $\sigma \beta_{p+i}\beta_{p+j}'\alpha_l$ with
$\alpha_l=(l, l+1, \cdots, p, p+1, p+2)$. So we have the following
one-to-one correspondence
$$S_{p, n-p} \times \{1\leq i < j \leq
 n-p\} \times \{1 \leq l \leq p \} \to S_{p-1, 2, 1, n-p-2},$$
 $$ (\sigma, i < j, l ) \mapsto \tau= \sigma \beta_{p+i}\beta_{p+j}'\alpha_l.$$
Note that $\sgn (\beta_{p+i}) =(-1)^{i-1}$, $\sgn (\beta_{p+j}')
=(-1)^{j}$, $\sgn (\alpha_l) =(-1)^{p-l}$ . Then
\begin{align*}
N_2=&(-1)^{\frac{p(p+1)}{2}}\sum_{\sigma \in S_{p, n-p}}\sgn (\sigma)
\sum_{1 \leq i < j \leq n-p} (-1)^{i+j} f(x_{\sigma(1)}, \cdots,
x_{\sigma(p)}) \otimes \sum_{l=1}^p \\
&\frac{\partial
{\{x_{\sigma(p+i)}, x_{\sigma(p+j)}\}}}{\partial x_{\sigma(l)}} \,
 \de\!x_{\sigma(l)} \wedge \de\!x_{\sigma(p+1)} \wedge
 \cdots \widehat{\de\!x_{\sigma(p+i)}} \cdots
\widehat{\de\!x_{\sigma(p+j)}} \cdots \wedge \de\!x_{\sigma(n)}\\
=&(-1)^{\frac{p(p+1)}{2}}\sum_{\substack{1 \leq i < j \leq n-p\\\sigma \in S_{p, n-p}}}
\sum_{l=1}^p \sgn (\sigma \beta_{p+i} \beta_{p+j}'\alpha_l)(-1)^{p-l-1} \\
&(-1)^{l-1} f(x_{\sigma \beta_{p+i} \beta_{p+j}'\alpha_l(p+2)},
x_{\sigma \beta_{p+i} \beta_{p+j}'\alpha_l(1)},
\cdots, x_{\sigma \beta_{p+i} \beta_{p+j}'\alpha_l (p-1)}) \\
&\otimes \frac{\partial\{x_{\sigma \beta_{p+i}
\beta_{p+j}'\alpha_l(p)}, x_{\sigma \beta_{p+i} \beta_{p+j}'\alpha_l
(p+1)}\}}{\partial x_{\sigma \beta_{p+i} \beta_{p+j}'\alpha_l
(p+2)}}\\
&\de\!x_{\sigma \beta_{p+i} \beta_{p+j}'\alpha_l(p+2)} \wedge \de\!x_{\sigma
\beta_{p+i} \beta_{p+j}'\alpha_l(p+3)}
\wedge \cdots \wedge \de\!x_{\sigma \beta_{p+i} \beta_{p+j}'\alpha_l(n)}\\
 =&(-1)^{\frac{p(p+1)}{2}+p}\sum_{\tau \in S_{p-1,2, 1,
n-p-2}}\sgn (\tau) f(x_{\tau(p+2)}, x_{\tau (1)}, \cdots,
x_{\tau (p-1)}) \otimes\\
&\frac{\partial\{x_{\tau (p)}, x_{\tau (p+1)}\}}{\partial x_{\tau (p+2)}} \,
 \de\!x_{\tau(p+2)} \wedge dx_{\tau(p+3)}\wedge \cdots \wedge
 \de\!x_{\tau(n)}.
\end{align*}
So, $V_2 = N_2$. Thus the lemma is proved.
\end{proof}

By lemma \ref{iso-chi-omega}, the commutative diagram in Proposition \ref{commu-dia} leads to a twisted Poincar\'{e} duality between
the Poisson homology of $R$ with values in $M_t$ and the Poisson cohomology of $R$ with values in $M$.
\begin{thm}\label{main-theorem}
For all $p \in \mathbb{N}$,  $\PH_p(R, M_t)\cong \PH^{n-p}(R, M)$.
The isomorphism is functorial.
\end{thm}

\begin{rmk}
 Suppose the polynomial Poisson algebra $R$ is unimodular, that is  $\phi_{\eta}=0$.
 Let $M=R$. Then $M_t=M=R$ and $\PH_p(R)\cong \PH^{n-p}(R)$, which is the  Poincar\'{e} duality
 for unimodular Poisson algebra.
\end{rmk}

\begin{rmk}  Theorem \ref{main-theorem}  generalizes \cite[Theorem
3.4.2]{LR} and \cite[Theorem 3.6]{Zhu} (See Example \ref{Launois+zhu}).
\end{rmk}

%

\section*{Acknowledgments}
Q.-S. Wu is supported by the NSFC (project 11171067 and key project 11331006), S.-Q. Wang is supported by the NSFC (project 11301180) , China Postdoctoral Science Foundation (project 2013M541478) and Fundamental Research Funds for the Central Universities (project 222201314032).


\begin{thebibliography}{99}

\bibitem[BP11]{BP} R. Berger, A. Pichereau, Calabi-Yau algebras viewed as
deformations of Poisson algebras, ArXiv: math.RA/1107.4472.


\bibitem[Bry88]{Bry} J. L. Brylinski, A differential complex for Poisson
manifolds, J. Differential Geom. 28 (1988), 93--114.



\bibitem[Dol09]{Do} V. A. Dolgushev, The Van den Bergh duality and the modular
symmetry of a Poisson variety, Selecta Math. 14 (2009), 199--228.

\bibitem [EG10]{EG} P. Etingof, V. Ginzburg, Noncommutative del Pezzo surfaces and Calabi-Yau algebras,
         arXiv: math. QA/0709.3593v4.


\bibitem[GMP93]{GMP}
J. Grabowski, G. Marmo, A. M. Perelomov, Poisson structures: towards a classification,
Mod. Phys. Lett. A 8 (1993), 1719--1733.




\bibitem[Hue90]{Hue} J. Huebschmann, Poisson cohomology and quantization, J.
Reine Angew. Math. 408 (1990), 57--113.

\bibitem[Hue99]{Hue1} J. Huebschmann, Duality for Lie¨CRinehart algebras and the modular class, J. Reine
Angew. Math. 510 (1999), 103--159.



\bibitem[LR07]{LR} S. Launois, L. Richard, Twisted Poincar\'{e} duality for some
quadratic Poisson algebras, Lett. Math. Phys. 79 (2007), 161--174.

\bibitem[Lic77]{Lic} A. Lichnerowicz, Les varieties de Poisson et leurs algebres
de Lie associees (French), J. Differential Geometry, 12 (1977),
253--300.

\bibitem[LWW14]{LWW} Twisted Poincar\'{e} duality for smooth affine Poisson varieties, preprint (2014).

\bibitem[Mar04]{Mar} N. Marconnet, Homologies of cubic Artin-Schelter regular
algebras, Journal of Algebra 278 (2004), 638--665.

\bibitem[Mas06]{Mas} T. Maszczyk, Maximal commutative subalgebras, Poisson
geometry and Hochschild homology, Preprint: ArXiv: math.KT/0603386.

\bibitem [Mon02]{Mon} P. Monnier,  Formal Poisson cohomology of quadratic Poisson structure,
         Lett. Math. Phys. 59, 253--267(2002).



\bibitem[Oh99]{Oh} S. Q. Oh, Poisson enveloping algebras, Comm. Algebra, 27
(1999), 2181--2186.









\bibitem [Pic06]{Pic} A. Pichereau, Poisson (co)homology and isolated singularities,
         Journal of Algebra 299 (2006), 747--777.

\bibitem [Prz01]{Prz} R. Przybysz, On one class of exact Poisson structures,
         Journal of Mathematical Physics 42 (2001), 1912--1920.

\bibitem [RV02]{RV} C. Roger, P. Vanhaecke, Poisson cohomology of the affine plane, Journal of Algebra
         251 (2002), 448--460.

\bibitem [TaP1]{TaP1} S. R. Tagne Pelap, On the Hochschild homology of elliptic Sklyanin algebras,
         Lett. Math. Phys. 87 (2009), 267--281 (arXiv: math. KT/0805.2852v5).

\bibitem [TaP2]{TaP2} S. R. Tagne Pelap, Poisson (co)homology of polynomial Poisson algebras in dimension four:
         Sklyanin's case, Journal of Algebra 322 (2009), 1151--1169 (arXiv: math. KT/0803.1635v1).

\bibitem [TaP3]{TaP3} S. R. Tagne Pelap, Homological properties of certain generalized
         Jacobian Poisson structures in dimension 3, J. Geom. Phys. 61 (2011), 2352--2368 (arXiv: math. QA/1004.3401V4).

\bibitem [VdB94]{VdB} M. Van den Bergh,  Noncommutative homology of some three-dimensional quantum
spaces, $K$-Theory 8 (1994),  213--230.



\bibitem [Wei77]{Wei0} A. Weinstein, Lecture on symplectic manifolds, CBMS Conference series in Math. 29, 1977.

\bibitem [Wei97]{Wei} A. Weinstein, The modular automorphism group of a Poisson manifold,
         J. Geom. Phys. 23 (1997), 379--394.

\bibitem[Xu99]{Xu} P. Xu, Gerstenhaber algebras and BV-algebras in Poisson
geometry, Commun. Math. Phys, 200 (1999), 545--560.

\bibitem[Ye00]{Ye} A. Yekutieli, The rigid dualizing complex of a universal enveloping algebra,
         Journal of Algebra 150 (2000), 85--93.

\bibitem [Zhu13]{Zhu} C. Zhu,  Twisted Poincar\'{e} duality for Poisson homology and cohomology
         of linear Poisson algebras (preprint 2012), to appear in Proc. Amer. Math. Soc.

\end{thebibliography}
\end{document}